\documentclass[11pt,a4paper]{article}
\usepackage[english]{babel}
\usepackage{latexsym}
\usepackage{amsmath,amsfonts,amsthm,amssymb,amscd}
\usepackage[mathscr]{eucal}
\usepackage{enumerate}
\usepackage[T1]{fontenc}
\usepackage[utf8]{inputenc}
\usepackage{graphicx}
\usepackage{fullpage}
\usepackage{url}
\usepackage{color}

\theoremstyle{definition}
\newcounter{Polya}

\newtheorem{theo}{Theorem}[section]
\newtheorem{pro}[theo]{Proposition}
\newtheorem{lemma}[theo]{Lemma}
\newtheorem{cor}[theo]{Corollary}
\newtheorem{defi}[theo]{Definition}
\newtheorem{obs}[theo]{Observation}
\newtheorem{rema}[theo]{Remark}
\newtheorem{prp}[theo]{Properties}
\newtheorem{exam}[theo]{Example}
\newtheorem{prps}[theo]{Propiedades}
\newtheorem{obss}[theo]{Observations}
\newtheorem{ejes}[theo]{Ejemplos}

\swapnumbers

\newenvironment{prop}{\smallskip\begin{pro}}{\end{pro}\smallskip}

\newenvironment{coro}{\smallskip\begin{cor}}{\end{cor}\smallskip}

\def\oo{\infty}
\def\ri{\rightarrow}
\def\a{\alpha}
\def\b{\beta}
\def\de{\delta}

\def\ep{\varepsilon}
\def\o{\omega}

\def\ga{\gamma}

\def\ro{\rho}

\def\ba{\boldsymbol{a}}

\newcommand{\PL}{Phragm\'{e}n-Lindel\"{o}f}
\def\en{\subseteq}

\def\N{\mathbb{N}}

\def\R{\mathbb{R}}
\def\C{\mathbb{C}}      

\def\Ri{{\cal R}}
\def\M{\mathbb{M}}

\def\bM{\mathbb{M}}
\def\L{\mathbb{L}} 

\def\m{\textbf{m}}

\def\bm{\textbf{m}}

\numberwithin{equation}{section}

\begin{document}

\title{A \PL\ theorem via proximate orders,\\
and the propagation of asymptotics}

\author{Javier Jiménez-Garrido, Javier Sanz and Gerhard Schindl}
\date{\today}

\maketitle

\abstract{We prove that, for asymptotically bounded holomorphic
functions in a sector in $\C$, an asymptotic expansion in a single direction towards the vertex with constraints in terms of a logarithmically convex sequence admitting a nonzero proximate order entails asymptotic expansion in the whole sector with control in terms of the same sequence. This generalizes a
result by A. Fruchard and C. Zhang for Gevrey asymptotic expansions, and the proof strongly rests on a suitably refined version of the classical \PL\ theorem, here obtained for functions whose growth in a sector is specified by a nonzero proximate order in the sense of E. Lindel\"of and G. Valiron.}


\section{Introduction}

In 1999, A. Fruchard and C. Zhang \cite{FruchardZhang} proved that for a holomorphic function in a sector $S$ which is bounded in every proper subsector of $S$, the existence of an asymptotic expansion following just one direction implies global (non-uniform) asymptotic expansion in the whole of $S$. Moreover, a Gevrey version of this result is provided with a control on the type, namely:

\begin{theo}[\cite{FruchardZhang}, Theorem 11]\label{theoFruchardZhang} Let $f$ be a function analytic and bounded in an open sector $S=S(d,\ga,r)$ of bisecting direction $d\in\R$, opening $\pi\ga$ and radius $r$, with $\ga,r>0$.
Suppose $f$ has asymptotic expansion $\hat{f}=\sum_{n=0}^\oo a_n z^n $ of Gevrey order $1/k$ ($k>0$) and type (at least) $R(\theta_0)>0$ in some direction $\theta_0$ with $|\theta_0-d|<\pi\ga/2$, i.e.,
for every $\de>0$ there exists $C=C(\de)>0$ such that for every $z\in S$ with $\arg(z)=\theta_0$ and every nonnegative integer $p$ we have that
\begin{equation}\label{eqDefinAsympExpanGevr}
|f(z)-\sum_{n=0}^{p-1}a_n z^n|\leq C\left(\frac{1}{R(\theta_0)}+\de\right)^p \Gamma(1+\frac{p}{k})|z|^p.
\end{equation}
Then, in every direction $\theta$ of $S$, $f$ admits $\hat{f}$ as its asymptotic expansion of Gevrey order $1/k$ and type $R(\theta)$ given as follows:
$$R(\theta)=  \left \{
        \begin{array}{rcrcrcr}
               R(\theta_0)  \left(\frac{\sin(k(\theta-\a))}{\sin (k(\a'-\a))}\right)^{1/k} & \text{if} &
							\theta\in ( \a, \a' ],   \\
							& &\\
             R(\theta_0) & \text{if} & \theta\in [\a', \b' ], \\
						& & \\
             R(\theta_0) \left(\frac{\sin(k(\theta-\b)}{\sin( k(\b'-\b))}\right)^{1/k} & \text{if} &
						\theta\in [\b', \b ).
        \end{array}
    \right . $$
Here, $\a=d-\pi\ga/2$ and $\b=d+\pi\ga/2$ are the directions of the radial boundary of $S$, $\a'=\min(\theta_0,\a+\frac{\pi}{2k} )\in(\a,\theta_0]$, and $\b'=\max(\theta_0, \b-\frac{\pi}{2k} )\in[\theta_0,\b)$.	
\end{theo}

We warn the reader that there is no agreement about the terminology in this respect: while most authors adhere, as we will do, to the convention that the asymptotics in \eqref{eqDefinAsympExpanGevr} is Gevrey of order $1/k$, others (for example, Fruchard and Zhang or W. Balser in~\cite{BalserLNM}) say this is of order $k$. Moreover, the notion of type is not standard, compare with the definition by M. Canalis-Durand~\cite{CanalisDurand} for whom the type in case one has \eqref{eqDefinAsympExpanGevr} is $(1/R+\de)^k$. It should also be mentioned that the factor $\Gamma(1+p/k)$ could be changed into $(p!)^{1/k}$ without changing the asymptotics, but this would affect the base of the geometric factor providing the type (by Stirling's formula, see \cite[pp.\ 3-4]{CanalisDurand}) in any case. As it will be explained below, our interest in the type will be limited, and so we will choose a simple approach in this respect, see Definitions~\ref{defiUniformMAsympExpan} and~\ref{defiMAsympExpanDirec}.

The proof of this result is based on the classical \PL\ theorem and on the so-called Borel-Ritt-Gevrey theorem. This last statement provides the surjectivity, as long as the opening of the sector is at most $\pi/k$, of the Borel map sending a function with Gevrey asymptotic expansion of order $1/k$ in a sector to its series of asymptotic expansion, whose coefficients will necessarily satisfy Gevrey-like estimates. Also, the injectivity of the Borel map in sectors of opening greater than $\pi/k$ (known as Watson's lemma) plays an important role when specifying conditions that guarantee the uniqueness of a function with a prescribed Gevrey asymptotic expansion of order $1/k$ in a direction.

The main aim of this paper is to extend these results for
other types of asymptotic expansions available in the literature. This possibility was already mentioned in \cite{LastraMozoSanz}, where
A. Lastra, J. Mozo-Fern\'andez and the second author of this paper generalized the results of Fruchard and Zhang for holomorphic functions of several variables in a polysector (cartesian product of sectors) admitting strong asymptotic expansion in the sense of H. Majima~\cite{Majima,Majima2}, considering also the Gevrey case as introduced by Y. Haraoka in~\cite{Haraoka}.

The asymptotics we will consider are those associated to the consideration of general ultraholomorphic classes of functions defined by constraining the growth of the sequence of their successive derivatives in a sector in terms of a sequence $\M=(M_p)_{p\in\N_0}$ of positive numbers ($\N_0=\{0,1,2,\ldots\}=\{0\}\cup\N$). This sequence will play the role of $(\Gamma(1+p/k))_{p\in\N_0}$ in \eqref{eqDefinAsympExpanGevr}, and it will be subject to precise conditions in order to guarantee not only the natural algebraic and analytic properties of the corresponding class, but the possibility of extending to this more general framework the results on the injectivity or surjectivity of the Borel map and a \PL-like statement.

For log-convex sequences $\M$ the considered ultraholomorphic classes are algebras, and the injectivity of the Borel map had been characterized in the 1950's by S. Mandelbrojt~\cite{Mandelbrojt} for uniform asymptotics (see Theorem~\ref{theo.Mandelbrojt} in this paper) and by B. Rodr{\'\i}guez-Salinas~\cite{Salinas} for uniformly bounded derivatives (see Theorem~\ref{Theo.Salinas} here). However, regarding surjectivity only some partial results were available by J. Schmets and M. Valdivia~\cite{SchmetsValdivia} and V. Thilliez~\cite{Thilliez} at the very beginning of this century, resting on results from the ultradifferentiable setting (dealing with classes of smooth functions in open sets of $\R^n$ with suitably controlled derivatives) and disregarding questions about the optimality of the opening of the sector or the variation of the type along with the direction in the sector. Moreover, the techniques used, of a functional-analytic nature, do not provide any insight into a possible extension of the \PL\ theorem.
However, the second author~\cite{SanzFlat} has recently made intervene the classical concept of proximate order in these concerns, making possible to obtain more precise statements concerning the injectivity and surjectivity of the Borel map. Subsequently, the authors~\cite{JimenezSanz,JimenezSanzSchindl} have studied the relationship between log-convex sequences, proximate orders and the property of regular variation. As a result, a deeper understanding has been gained of the property of admissibility of a proximate order by a log-convex sequence, which gives the key for obtaining in this paper an analogue of \PL\ theorem for functions whose growth in a sector is specified in terms of such a sequence $\M$. It is worth mentioning that sequences admitting a proximate order are strongly regular (in the sense of Thilliez), and that all the instances of strongly regular sequences appearing in applications do admit a proximate order.

As in the Gevrey case, the study of the type as the direction moves in the sector is possible, although some information is lost in general (see Remark~\ref{rema.bisect.direct.Mflat}). This is due to the fact that the classical exponential kernel appearing in the finite Laplace transform providing the solution of the Borel-Ritt-Gevrey theorem in the Gevrey case is now replaced by the exponential of a function whose behavior at infinity is only given by some asymptotic relations, which is not enough for an accurate handling of the resulting type. However, in case the sequence $\M$ not only admits a proximate order, but provides one, the type may be better described.

The paper is organized as follows.
After fixing some notations, Section 2 is devoted to some preliminaries on general asymptotic expansions, ultraholomorphic classes and quasianalyticity results, specially when proximate orders are available.
All this material will be needed in Section 3, where several lemmas of a \PL\ flavor are obtained. A paradigm is Lemma~\ref{lemma.Mflat.one.direction.small.opening}, where exponential decrease is extended from just one direction to a whole small (in the sense of its opening) sector adjacent to it. Section 4 contains several versions of Watson's lemma on the uniqueness of a function admitting a given asymptotic expansion in a direction, and in the final Section 5 we characterize the functions with an asymptotic expansion in a sectorial region as those asymptotically bounded and admitting such expansion in just one direction in the region.

\section{Preliminaries}

We set $\N:=\{1,2,...\}$, $\N_{0}:=\N\cup\{0\}$.
$\mathcal{R}$ stands for the Riemann surface of the logarithm.
We consider bounded \emph{sectors}
$$S(d,\gamma,r):= \{z\in\mathcal{R}:|\hbox{arg}(z)-d|<\frac{\gamma\,\pi}{2},\ |z|<r\},
$$
respectively unbounded sectors
$$
S(d,\gamma):=\{z\in\mathcal{R}:|\hbox{arg}(z)-d|<\frac{\gamma\,\pi}{2}\},
$$
with \emph{bisecting direction} $d\in\R$, \emph{opening} $\gamma \,\pi$ ($\gamma>0$) and (in the first case) \emph{radius} $r\in(0,\infty)$. For unbounded sectors of opening $\gamma\,\pi$ bisected by direction 0, we write
$S_{\gamma}:=S(0,\gamma)$. In some cases, it will also be convenient to consider sectors whose elements have their argument in a half-open, or in a closed, bounded interval of the real line.\par\noindent

A \emph{sectorial region} $G(d,\gamma)$ with bisecting direction $d\in\R$ and opening $\gamma\,\pi$ will be an open connected set in $\mathcal{R}$ such that $G(d,\ga)\subset S(d,\ga)$, and
for every $\beta\in(0,\gamma)$ there exists $\rho=\rho(\beta)>0$ with $S(d,\beta,\rho)\subset G(d,\gamma)$.
In particular, sectors are sectorial regions. If $d=0$ we just write $G_\ga$.\par\noindent

A bounded (respectively, unbounded) sector $T$ is said to be a \emph{proper subsector} of a sectorial region (resp. of an unbounded sector) $G$, and we write $T\ll G$ (resp. $T\prec G$), if $\overline{T}\subset G$ (where the closure of $T$ is taken in $\mathcal{R}$, and so the vertex of the sector is not under consideration).

For an open set $U\subset\mathcal{R}$, the set of all holomorphic functions in $U$ will be denoted by~$\mathcal{H}(U)$. 
$\C[[z]]$ stands for the set of formal power series in $z$ with complex coefficients.

\subsection{Log-convex sequences and ultraholomorphic classes}

In what follows, $\M=(M_p)_{p\in\N_0}$ always stands for a sequence of positive real numbers, and we always assume that $M_0=1$.

\begin{defi}
We say a holomorphic function $f$ in a sectorial region $G$ admits the formal power series $\hat{f}=\sum_{n=0}^{\infty}a_{n}z^{n}\in\C[[z]]$ as its $\bM-$\emph{asymptotic expansion} in $G$ (when the variable tends to 0) if for every $T\ll G$ there exist $C_T,A_T>0$ such that for every $p\in\N_0$ one has
\begin{equation*}\Big|f(z)-\sum_{n=0}^{p-1}a_nz^n \Big|\le C_TA_T^pM_{p}|z|^p,\qquad z\in T.
\end{equation*}
We will write $f\sim_{\bM}\hat{f}$ in $G$, and $\tilde{\mathcal{A}}_{\M}(G)$ will stand for the space of functions admitting $\bM-$asymptotic expansion in $G$.
\end{defi}

\begin{defi}\label{defiUniformMAsympExpan}
Given a sector $S$, we say $f\in\mathcal{H}(S)$ admits $\hat{f}=\sum_{n=0}^{\infty}a_{n}z^{n}\in\C[[z]]$ as its \emph{uniform $\bM-$asymptotic expansion in $S$ (of type $1/A$ for some $A>0$}) if there exists $C>0$ such that for every $p\in\N_0$ one has
\begin{equation*}\Big|f(z)-\sum_{n=0}^{p-1}a_nz^n \Big|\le CA^pM_{p}|z|^p,\qquad z\in S.
\end{equation*}
$\tilde{\mathcal{A}}^u_{\M}(S)$ stands for the space of functions admitting uniform $\bM-$asymptotic expansion in $S$ (of some type).
\end{defi}

\begin{defi}
Given $\M=(M_p)_{p\in\N_0}$, a constant $A>0$ and a sector $S$, we define
$$\mathcal{A}_{\M,A}(S)=\big\{f\in\mathcal{H}(S):\left\|f\right\|_{\M,A}:=\sup_{z\in S,n\in\N_{0}}\frac{|f^{(p)}(z)|}{A^{p}p!M_{p}}<\infty\big\}.$$
($\mathcal{A}_{\M,A}(S),\left\| \ \right\| _{\M,A}$) is a Banach space, and $\mathcal{A}_{\M}(S):=\cup_{A>0}\mathcal{A}_{\M,A}(S)$ is called a \textit{Roumieu-Carleman ultraholomorphic class} in the sector $S$.
\end{defi}

Since the derivatives of $f\in\mathcal{A}_{\bM,A}(S)$ are Lipschitzian, for every $n\in\N_{0}$ one may define
$$f^{(n)}(0):=\lim_{z\in S,z\to0 }f^{(n)}(z)\in\C.$$
We recall now the relationship between these classes and the concept of asymptotic expansion. As a consequence of Taylor's formula,
we have the following result (see \cite{BalserLNM,galindosanz}).

\begin{prop}\label{prop.ultra.class.inclusion}
Let $S$ be a sector, if $f\in\mathcal{A}_{\M,A}(S)$ then $f$ admits $\hat{f}=\sum_{p\in\N_0}\frac{1}{p!}f^{(p)}(0)z^p$ as its uniform $\bM-$asymptotic expansion in $S$ of type $1/A$. Consequently, we have that
$$\mathcal{A}_{\M}(S)\en \tilde{\mathcal{A}}^u_{\M}(S) \en  \tilde{\mathcal{A}}_{\M}(S).$$
\end{prop}

Next we specify some conditions on the sequence $\M$ that will have important consequences on the previous classes or spaces.

\begin{defi}\label{propiedsuces}
We say:
\begin{itemize}
\item[(i)]  $\bM$ is \textit{logarithmically convex} (for short, (lc)) if
$$M_{p}^{2}\le M_{p-1}M_{p+1},\qquad p\in\N.
$$
\item[(ii)]  $\bM$ is \textit{derivation closed} (for short, (dc)) if there exists $A>0$ such that
$$M_{p+1}\le A^{p+1}M_{p},\qquad p\in\N_0.$$
\item[(iii)]  $\bM$ is of \textit{moderate growth} (briefly, (mg)) if there exists $B>0$ such that
$$M_{p+q}\le B^{p+q}M_{p}M_{q},\qquad p,q\in\N_0.$$
\item[(iv)]  $\bM$ satisfies the \textit{strong non-quasianalyticity condition} (for short, (snq)) if there exists $C>0$ such that
$$
\sum_{q\ge p}\frac{M_{q}}{(q+1)M_{q+1}}\le C\frac{M_{p}}{M_{p+1}},\qquad p\in\N_0.
$$
\end{itemize}
\end{defi}

Obviously, (mg) implies (dc).

\begin{defi}[V. Thilliez~\cite{Thilliez}]\label{sucfuereg}
We say $\bM$ is \textit{strongly regular} if it verifies (lc), (mg) and (snq).
\end{defi}

\begin{defi} For a sequence $\M$ we define {\it the sequence of quotients}  $\m=(m_p)_{p\in\N_0}$ by
$$m_p:=\frac{M_{p+1}}{M_p}, \qquad p\in \N_0.$$
\end{defi}

It is obvious that $\M$ is (lc) if, and only if, $\bm$ is nondecreasing.

\begin{defi} Let $\M$ and $\L$ be sequences, we say that {\it $\M$ is equivalent to $\L$}, and we write
 $\M\approx\L$, if there exist positive constants $A,B>0$ such that
$$A^p L_p\leq M_p\leq B^p L_p,\qquad p\in\N_0.$$
 \end{defi}

\begin{exam}\label{exampleSequences}
We mention some interesting examples. In particular, those in (i) and (iii) appear in the applications of summability theory to the study of formal power series solutions for different kinds of equations.
\begin{itemize}
\item[(i)] The sequences $\M_{\a,\b}:=\big(p!^{\a}\prod_{m=0}^p\log^{\b}(e+m)\big)_{p\in\N_0}$, where $\a>0$ and $\b\in\R$, are strongly regular (more precisely, in case $\b<0$ the sequence is equivalent to a strongly regular one, see Remark~\ref{rema.alg.prop.ultra.class}).
    For $\b=0$, we have the best known example of strongly regular sequence, $\M_{\a,0}=(p!^{\a})_{p\in\N_{0}}$, called the \textit{Gevrey sequence of order $\a$}.
\item[(ii)] The sequence $\M_{0,\b}:=(\prod_{m=0}^p\log^{\b}(e+m))_{p\in\N_0}$, with $\b>0$, is (lc), (mg) and $\bm$ tends to infinity, but (snq) is not satisfied.
\item[(iii)] For $q>1$, $\M_q:=(q^{p^2})_{p\in\N_0}$ is (lc) and (snq), but not (mg).
\end{itemize}
\end{exam}

\begin{rema}\label{rema.alg.prop.ultra.class} For any sequence $\M$, the classes $\mathcal{A}_{\M}(S)$, $\tilde{\mathcal{A}}^u_{\M}(S)$ and $\tilde{\mathcal{A}}_{\M}(S)$ are vector spaces. If $\M$ is (lc), they are algebras; if $\M$ is (dc), they are stable under taking derivatives.
Moreover, equivalent sequences define the same classes.
\end{rema}

\begin{defi}\label{defiMAsympExpanDirec}
Let $f$ be a function defined in a sectorial region $G=G(d,\ga)$, and $\theta$ be a direction in $G$, i.e. $|\theta-d|<\pi\ga/2$. We say
$f$ has \emph{$\M$-asymptotic expansion $\hat{f}=\sum_{n=0}^\oo a_n z^n $ in direction $\theta$} if
there exist $r_\theta,C_\theta,A_\theta>0$ such that the segment $(0,r_\theta e^{i\theta}]$ is contained in $G$, and for every $z\in (0,r_\theta e^{i\theta}]$ and every $p\in\N_0$ one has
\begin{equation*}
\Big|f(z)-\sum_{n=0}^{p-1}a_nz^n \Big|\le C_\theta A_\theta^pM_{p}|z|^p.
\end{equation*}
In this case, we say the type is $1/A_\theta$. Of course, the definition makes sense as long as the function is defined only in direction $\theta$ near the origin, i.e. in a segment $(0,r e^{i\theta}]$ for suitable $r>0$.
\end{defi}

One may accordingly define classes of formal power series
$$\C[[z]]_{\M,A}=\Big\{\hat{f}=\sum_{n=0}^\oo a_nz^n\in\C[[z]]:\, \left|\,\ba \,\right|_{\M,A}:=\sup_{p\in\N_{0}}\displaystyle \frac{|a_{p}|}{A^{p}M_{p}}<\infty\Big\}.$$
$(\C[[z]]_{\M,A},\left| \  \right|_{\M,A})$ is a Banach space, and we put $\C[[z]]_{\M}:=\cup_{A>0}\C[[z]]_{\M,A}$.

\begin{rema}\label{rema.limit.f.at.0}
Given $f\in\tilde{\mathcal{A}}_{\M}(G)$ with $f\sim_{\M}\hat{f}=\sum_{p=0}^{\infty}a_{p}z^{p}$, it is plain to check that for every bounded proper subsector $T$ of $G$ and every $p\in\N_0$ one has
$$ a_p=\lim_{ \genfrac{}{}{0pt}{}{z\ri0}{z\in T}} \frac{f^{(p)}(z)}{p!},$$
and we can set ${f^{(p)}(0)}:=p!a_p$. Moreover, if we define $\tilde{\mathcal{B}}(f):=\hat{f}$, it is straightforward that $\tilde{\mathcal{B}}(f)\in\C[[z]]_{\M}$, and the map $\tilde{\mathcal{B}}:\tilde{\mathcal{A}}_{\M}(G)\longrightarrow \C[[z]]_{\M}$
so defined is the \textit{asymptotic Borel map}. If $S$ is a sector, using Proposition~\ref{prop.ultra.class.inclusion}
we see that the asymptotic Borel map is also well defined on $\mathcal{A}_{\M}(S)$ and $\tilde{\mathcal{A}}^u_{\M}(S)$.
\end{rema}

\subsection{Classical quasianalyticity results}\label{subsectClassicalQuasianal}

We introduce first the notions of flatness and quasianalyticity.

\begin{defi}
A function $f$ in any of the previous classes is said to be \textit{flat} if $\tilde{\mathcal{B}}(f)$ is the null formal power series (denoted $\hat{0}$), or in other words, $f\sim_{\M}\hat{0}$.
\end{defi}

\begin{defi}
Let $S$ be a sector, $G$ a sectorial region and $\M=(M_{p})_{p\in\N_{0}}$ be a sequence of positive numbers. We say that $\mathcal{A}_{\M}(S)$, $\tilde{\mathcal{A}}^u_{\M}(S)$,  or $\tilde{\mathcal{A}}_{\M}(G)$ is  \textit{quasianalytic} if it does not contain nontrivial flat functions (in other words, the Borel map is injective in this class).
\end{defi}

In order to simplify some statements or to avoid trivial situations, from now on in this paper we will assume the standard property that
$$
\textbf{the sequence $\M$ is logarithmically convex with $\lim_{p\ri\oo} m_p=\oo$}.
$$
The following result characterizes quasianalyticity for the classes of functions with uniformly bounded derivatives in an unbounded sector.
It first appeared in Rodr{\'\i}guez-Salinas~\cite{Salinas}, although it is frequently attributed to B. I. Korenbljum~\cite{korenbljum}.

\begin{theo}[\cite{Salinas}, \ Theorem 12]\label{Theo.Salinas}
Let $\M$ and $\ga>0$ be given. The following statements are equivalent:
\begin{itemize}
\item[(i)] The class $\mathcal{A}_{\M}(S_{\gamma})$ is quasianalytic.
\item[(ii)] $\displaystyle \sum_{p=0}^{\infty}\Big(\frac{1}{(p+1)m_{p}}\Big)^{1/(\ga+1)}$ diverges.
\end{itemize}
\end{theo}

This result can be rewritten in terms of the classical notion of exponent of convergence of a sequence.

\begin{prop}[\cite{holland}, p.\ 65]
Let $(c_n)_{n\in\N_0}$ be a nondecreasing sequence of positive real numbers tending to infinity. The \textit{exponent of convergence} of $(c_n)_n$ is defined as
$$
\lambda_{(c_n)}:=\inf\{\mu>0:\sum_{n=0}^\infty \frac{1}{c_n^{\mu}}\textrm{ converges}\}
$$
(if the previous set is empty, we put $\lambda_{(c_n)}=\infty$). Then, one has
\begin{equation*}
\lambda_{(c_n)}=\limsup_{n\to\infty}\frac{\log(n)}{\log(c_n)}.
\end{equation*}
\end{prop}

According to this last formula, we may define the index
$$\o(\M):=\liminf_{p\to\infty} \frac{\log(m_p)}{\log(p)},$$
in such a way that
\begin{equation}\label{eqFormulaExponentConvergence}
\lambda_{(m_p)}=\frac{1}{\o(\M)}, \qquad \lambda_{((p+1)m_p)}=\frac{1}{\o(\M)+1}.
\end{equation}

So, Theorem~\ref{Theo.Salinas} may be stated as

\begin{coro}\label{Coro.Salinas}
Let $\M$ and $\ga>0$ be given.
The following statements are equivalent:
\begin{itemize}
\item[(i)] The class $\mathcal{A}_{\M}(S_{\gamma})$ is quasianalytic.
\item[(ii)] $\ga>\o(\M)$, or $\ga=\o(\M)$ and $\displaystyle \sum_{p=0}^{\infty}\Big(\frac{1}{(p+1)m_{p}}\Big)^{1/(\o(\M)+1)}$ diverges.
\end{itemize}
\end{coro}

\begin{rema}\label{remaTheorSalinasBoundedSectors}
The problem of quasianalyticity for classes of functions with uniformly bounded derivatives in bounded regions has also been treated. In the works of K. V. Trunov and R. S. Yulmukhametov~\cite{Yulmukhametov,TrunovYulmukhametov} a characterization is given, for a convex bounded region containing 0 in its boundary, in terms of the sequence $\M$ and of the way the boundary approaches 0. In particular, for bounded sectors, if $\gamma\le 1$, $d\in\R$ and $r>0$, it turns out that the class $\mathcal{A}_{\M}(S(d,{\gamma},r))$ is quasianalytic precisely when condition (ii) above is satisfied.
\end{rema}

The study of quasianalyticity for the classes of functions with uniform $\M$-asymptotic expansion in an unbounded sector rests on the following statement by B. I. Mandelbrojt.

\begin{theo}[\cite{Mandelbrojt},\ Section\ 2.4.III]\label{theo.Mandelbrojt}
 Let $\M$ be given, $H=\{z\in\C: \Re(z)>0 \}$ and $\ga>0$. The following statements are equivalent:
 \begin{enumerate}[(i)]
  \item If $f\in \mathcal{H}(H)$ and there exist $A,C>0$ such that
  \begin{equation}\label{eqBoundsTheorMandel}
  |f(z)|\leq \frac{CA^pM_p}{|z|^{\ga p}}, \qquad z\in H, \quad p\in \N_0,
  \end{equation}
  then $f$ identically vanishes.
  \item $\displaystyle \sum_{p=0}^{\oo} \left(\frac{1}{m_p} \right)^{1/\ga}$ diverges.
 \end{enumerate}
\end{theo}

Observe that a function $f$ is holomorphic in $H$ and verifies the estimates \eqref{eqBoundsTheorMandel} if, and only if, the function $g$ given by $g(z):=f(1/z^{1/\ga})$ belongs to $\tilde{\mathcal{A}}_{\M}^u(S_{\ga})$ and is flat. From this fact and the first equality in \eqref{eqFormulaExponentConvergence}, it is immediate to deduce the next characterization.

\begin{coro}[generalized Watson's lemma for uniform asymptotics] \label{corollary.Watson.Lemma.Uniform.Bounds}
Let $\M$ and $\ga>0$ be given. The following are equivalent:
\begin{enumerate}[(i)]
 \item $\mathcal{\tilde{A}}^u_{\M}(S_\ga)$ is quasianalytic.
 \item $\displaystyle \sum_{p=0}^{\infty}\Big(\frac{1}{m_{p}}\Big)^{1/\ga}$ diverges.
 \item $\ga>\o(\M)$, or $\ga=\o(\M)$ and $\displaystyle \sum_{p=0}^{\infty}\Big(\frac{1}{m_{p}}\Big)^{1/\o(\M)}$ diverges.
\end{enumerate}
\end{coro}

\begin{rema}\label{remaQuasianalUniformAsymptBoundedSectors}
This theorem holds true for bounded sectors $S(0,\ga,r)$ with similar arguments. Proceeding as in~\cite[Theorem\ 2.19]{JimenezSanz}, we only need to modify the proof of (ii)$\Rightarrow$(i) by considering the transformation $w(z)=1/(z + (1/r)^{1/\ga})^{\ga} $, which maps $H$ into a region $D$ contained in $S(0,\ga,r)$: given a flat function $g\in  \mathcal{\tilde{A}}^u_{\M}(S(0,\ga,r))$, the function $f(w):=g(z(w))$ is defined in $H$ and, by Mandelbrojt's theorem, it identically vanishes.
\end{rema}

Regarding the class of functions with (non-uniform) asymptotic expansion in a sectorial region $G$, we first express flatness in $\tilde{\mathcal{A}}_{\M}(G)$ by means of an auxiliary function: For $t>0$ we define
\begin{equation*}\label{equadefiMdet}
M(t):=\sup_{p\in\N_{0}}\log\big(\frac{t^p}{M_{p}}\big)=\left \{ \begin{matrix}  p\log t -\log(M_{p}) & \mbox{if }t\in [m_{p-1},m_{p}),\ p=1,2,\ldots,\\
0 & \mbox{if } t\in [0,m_{0}), \end{matrix}\right.
\end{equation*}
which is a non-decreasing continuous map in $[0,\infty)$ with $\lim_{t\to\infty}M(t)=\infty$. Then, we have the following result.

\begin{theo}[\cite{ThilliezSmooth}, Proposition 4]\label{theo.thillez.flat.function.M}
Given $f\in\mathcal{H}(G)$, the following are equivalent:
\begin{enumerate}[(i)]
 \item $f\in\tilde{\mathcal{A}}_{\M}(G)$ and $f$ is flat.
 \item For every bounded proper subsector $T$ of $G$ there exist $c_1,c_2>0$ with
 $$|f(z)|\le c_1e^{-M(1/(c_2|z|))},\qquad z\in T. $$
\end{enumerate}
\end{theo}

\begin{rema}\label{rema.bounded.function.null.asymp.expan.dir}
In the conditions of Definition~\ref{defiMAsympExpanDirec}, if $\hat{f}$ is the null series we say that $f$ is $\M$-\emph{flat in direction} $\theta$. As in the previous statement, this amounts to the existence of $r_\theta,C_\theta,A_\theta>0$ such that the segment $(0,r_\theta e^{i\theta}]$ is contained in $G$, and for every $z\in (0,r_\theta e^{i\theta}]$ one has
\begin{equation*}
|f(z)|\le C_\theta e^{-M(1/(A_\theta |z|))}.
\end{equation*}
Suppose moreover that $f$ is bounded throughout the (bounded or not) sectorial region $G$. Since the function  $e^{-M(t)}$ is non-increasing in $[0,\infty)$, it is obvious that $f$ is $\M$-flat in direction $\theta$ if, and only if, there exist $\tilde{C}_\theta>0$ and the same constant $A_\theta>0$ as before, such that for every $z\in G$ with $\arg(z)=\theta$ one has
\begin{equation}\label{eqMFlatDirecBoundedFunction}
|f(z)|\le \tilde{C}_\theta e^{-M(1/(A_\theta |z|))}.
\end{equation}
This fact will be used later on.
\end{rema}

\subsection{Quasianalyticity results via proximate orders}\label{subsectQuasianalResultsProxOrders}

An easy characterization of quasianalyticity in the classes $\tilde{\mathcal{A}}_\M(G)$ may be given thanks to the notion of proximate order, appearing in the theory of
growth of entire functions and developed, initially, by E. Lindel\"of and G. Valiron. We will focus our discussion
mainly on the results given by L. S. Maergoiz (see~\cite{Maergoiz}).

\begin{defi}\label{OAD:1}
We say a real function $\ro(r)$, defined on $(c,\oo)$ for some $c\ge 0$, is a {\it proximate order},
if the following hold:
 \begin{enumerate}[(A)]
  \item $\ro$ is continuous and piecewise continuously differentiable in $(c,\oo)$
  (meaning that it is differentiable except possibly at a sequence of points, tending to infinity, at any
  of which it is continuous and has distinct finite lateral derivatives),\label{OA1:1}
  \item $\ro(r) \geq 0$ for every $r>c$,\label{OA2:1}
  \item $\lim_{r \ri \oo} \ro(r)=\ro< \oo$, \label{OA3:1}
  \item $\lim_{r  \ri \oo} r \ro'(r) \ln r = 0$. \label{OA4:1}
 \end{enumerate}
 In case the value $\rho$ in (C) is positive (respectively, is 0), we say $\rho(r)$ is a \textit{nonzero}
 (resp. \textit{zero}) proximate order.
\end{defi}

\begin{rema}\label{rema.V.order.less.ro}
If $\ro(r)$ is a proximate order with limit $\ro$ at infinity  and $\tau>\ro$, then there exists $r(\tau)>1$ such that
$\ro(r)<\tau$ for $r>r(\tau)$ and, consequently,
$$r^{\ro(r)}<r^{\tau}, \qquad r>r(\tau).$$
\end{rema}

We now associate to a nonzero proximate order a class of functions with nice properties, which will play a prominent role in our \PL\ result.

\begin{theo}[\cite{Maergoiz}, Theorem\ 2.4]\label{propanalproxorde}
Let $\ro(r)$ be a nonzero proximate order such that $\lim_{r \ri \oo} \ro(r)=\ro$. For every $\ga>0$ there exists an analytic
function $V(z)$ in $S_\ga$ such that:
  \begin{enumerate}[(I)]
  \item  For every $z \in S_\ga$,
 \begin{equation*}
    \lim_{r \to \infty} \frac{V(zr)}{V(r)}= z^{\ro},
  \end{equation*}
uniformly in the compact sets of $S_\ga$.
\item $\overline{V(z)}=V(\overline{z})$ for every $z \in S_\ga$ (where, for $z=(|z|,\arg(z))$, we put $\overline{z}=(|z|,-\arg(z))$).
\item $V(r)$ is positive in $(0,\infty)$, monotone increasing and $\lim_{r\to 0}V(r)=0$.
\item The function $t\in\R\to V(e^t)$ is strictly convex (i.e. $V$ is strictly convex relative to $\log(r)$).
\item The function $\log(V(r))$ is strictly concave in $(0,\infty)$.
\item  The function $\ro_V(r):=\log( V(r))/\log(r)$, $r>0$, is a proximate order equivalent to $\ro(r)$, i.e.,
$$\lim_{r\ri\oo} V(r)/r^{\ro(r)}=\lim_{r\ri\oo} r^{\ro_V(r)} / r^{\ro(r)} = 1.$$
    \end{enumerate}
\end{theo}

Given $\ga>0$ and $\ro(r)$ as before, $MF(\ga,\ro(r))$ will denote the set of Maergoiz functions $V$ defined in $S_{\ga}$ and satisfying the conditions (I)-(VI) of Theorem~\ref{propanalproxorde}.

Before returning to the study of quasianalyticity, we indicate how to go from sequences to proximate orders (for more information on this relation and its reversion, see~\cite{JimenezSanzSchindl}). Given $\M$ and its associated function $M(t)$, for $t$ large enough we can consider $$
d_{\M}(t):=\log(M(t))/\log(t).
$$
The following result characterizes those sequences for which $d_{\M}(t)$ is a proximate order.

\begin{theo}[\cite{JimenezSanzSchindl},\, Theorem 3.6]\label{theorem.charact.prox.order.nonzero}
Let $\M$ be given. The following are equivalent:
\begin{enumerate}[(a)]
 \item $d_{\M}(t)$ is a proximate order with $\lim_{t\to\infty}d_{\M}(t)\in(0,\infty)$.
 \item There exists $\lim_{p\to\infty} \log\big(m_p/M_p^{1/p}\big)\in(0,\infty)$.
 \item $\m$ is regularly varying with a positive index of regular variation.
 \item There exists $\o>0$ such that for every natural number $\ell\geq 2$,
 $$\lim_{p\ri\oo} \frac{m_{\ell p}}{m_p}=\ell^\o.$$
\end{enumerate}
In case any of these statements holds, the value of the limit mentioned in (b), that of the index mentioned in (c), and that of the constant $\o$ in (d) is $\o(\M)$,
and the limit in (a) is $1/\o(\M)$.
\end{theo}

A less restrictive condition on the sequence $\M$, namely the admissibility of a proximate order, is indeed sufficient for our purposes.

\begin{theo}[\cite{JimenezSanzSchindl},\, Theorem 4.14]\label{theo.charact.admit.p.o}
Given $\M$, 
the following conditions are equivalent:
\begin{enumerate}
 \item[(e)] There exists a (lc) sequence $\L$, with quotients tending to infinity, such that $\L\approx\M$ and $d_{\L}(t)$ is a nonzero proximate order.
 \item[(f)] $\M$ \emph{admits a nonzero proximate order}, i.e., there exist a nonzero proximate order $\ro(t)$
 and constants $A$ and $B$ such that
\begin{equation}\label{eqAdmitsProximateOrder}
A\leq \log(t)(\ro(t)-d_{\M}(t)) \leq B, \qquad t\textrm{ large enough}.
\end{equation}
\end{enumerate}
\end{theo}

From this result, we deduce that whenever a class $\tilde{\mathcal{A}}_\M(G)$ (or $\tilde{\mathcal{A}}^u_\M(S)$ or $\mathcal{A}_\M(S)$) is defined in terms of a 
sequence $\M$ admitting a nonzero proximate order, we can exchange $\M$ by another equivalent (lc) sequence $\L$ whose sequence of quotients is regularly varying. Then, we can briefly say that the $\M$-asymptotic expansion of a function
$f\in\tilde{\mathcal{A}}_\M(G)=\tilde{\mathcal{A}}_\L(G)$ has log-convex regularly varying constraints.

\begin{rema}\label{rema.order.ro.1.over.omega}
If $\M$ 
admits a nonzero proximate order $\ro(t)$, it is clear that
$\lim_{t\to\oo}d_{\M}=\lim_{t\to\oo} \ro(t)$ (see~\eqref{eqAdmitsProximateOrder}), and from~\cite[Remark\ 4.15]{JimenezSanzSchindl} we deduce that this common value is $1/\o(\M)$.
\end{rema}

\begin{rema}\label{rema.M.and.V.admsibility.condition}
If $\M$ admits a nonzero proximate order $\ro(r)$, for every $\ga>0$, thanks to (VI), we know that there exist $V\in MF(\ga,\ro(r))$
 and positive constants $A,B,t_0$ such that
\begin{equation}\label{ineq.function.M.and.V}
AV(t)\leq M(t)\leq BV(t), \qquad  t>t_0.
\end{equation}
\end{rema}

In \cite[Remark~4.15]{JimenezSanzSchindl} it has been shown that sequences admitting a proximate order are indeed strongly regular. So, as indicated in~\cite[Remark 4.11.(iii)]{SanzFlat}, for such sequences $\M$ one may construct nontrivial flat functions in $\tilde{\mathcal{A}}_{\M}(G_{\o(\M)})$, what allows us to state the following version of Watson's Lemma for non-uniform asymptotics.

\begin{theo}[\cite{SanzFlat}, Corollary~4.12]\label{TheoWatsonlemmaProxOrder}
Suppose $\M$ 
admits a nonzero proximate order, and let $\ga>0$ be given. The following statements are equivalent:
\begin{itemize}
\item[(i)] $\tilde{\mathcal{A}}_{\M}(G_{\gamma})$ is quasianalytic.
\item[(ii)] $\gamma>\omega(\M)$.
\end{itemize}
\end{theo}

Moreover, for such sequences we can generalize Borel-Ritt-Gevrey theorem \cite{SanzFlat} and the Gevrey summability theory following Balser's moment summability methods, see~\cite{lastramaleksanz3}.

\begin{rema}\label{rema.Watson.Lemma.bisect.direct}
Corollary~\ref{Coro.Salinas} , Corollary~\ref{corollary.Watson.Lemma.Uniform.Bounds} and Theorem~\ref{TheoWatsonlemmaProxOrder} are also valid if we change the bisecting direction of the considered sectorial region.
\end{rema}

\section{$\M$-flatness extension}\label{sectionFlatnessExtension}

From this point on we will assume not only that the sequence $\M$ is logarithmically convex with $\lim_{p\ri\oo} m_p=\oo$, but also that
$$
\textbf{the sequence $\M$ admits a nonzero proximate order}.
$$
This is not a strong assumption for strongly regular sequences, since it is satisfied by every such sequence appearing in applications (the Gevrey ones, or the one associated to the $1^+$-level asymptotics). However, note that there are strongly regular sequences which do not satisfy it, see~\cite{JimenezSanzSchindl}.

We are ready for proving an important lemma about the extension of $\M$-flatness from a boundary direction into a whole small sector for functions bounded there and admitting a continuous extension to the boundary (considered in $\mathcal{R}$, i.e., disregarding the origin). First, we recall a classical version of \PL\ theorem needed in the proof.

\begin{theo}[\PL\ theorem, \cite{Titchmarsh}, p.\ 177]\label{Theo.Prag.Lind}
 Let $f$ be a function holomorphic in a sector $S=S(d,\ga,\ro)$, continuous and bounded by $C$ in the boundary $\partial S$.
  Suppose there exist $K,L>0$ and $\o>\ga$ such that
 $$|f(z)|<Ke^{L |z|^{-1/\o}}$$
 for every $z\in S$. Then $f$ is bounded by $C$ in the sector $S$.
\end{theo}

Now we obtain an analogue of \PL\ theorem for $\M$-flat functions in a sector.


\begin{lemma}
\label{lemma.Mflat.one.direction.small.opening}
Let $\M$ and $0<\ga<\o(\M)$ be given. Suppose $f$ is a bounded holomorphic function in $S_{\ga}$ that admits a continuous extension to the boundary $\partial S_\ga$, and that is $\M$-flat in direction $d=\pi\ga/2$.
Then for every $0<\de<\pi\ga$, there exist constants $k_1(\de),k_2(\de)>0$ with
 $$|f(z)|\leq k_1e^{-M(1/(k_2|z|))},\qquad \arg(z)\in[-\pi\ga/2+\de,\pi\ga/2].  $$
\end{lemma}

\begin{proof}
 For simplicity, we denote $\o:=\o(\M)$.  We fix $0<\de<\pi\ga$. Since $\ga<\o$, we have that
 $$\frac{\pi}{2}<\b=\b(\de):=\frac{1}{\o} (\frac{\pi}{2}\o+\frac{\de}{2}) <  \pi, \quad
  -\frac{\pi}{2}+\frac{\de}{2\o}<\a=\a(\de):=\frac{1}{\o} (\frac{\pi}{2}\o-\pi\ga+\frac{\de}{2}) < \frac{\pi}{2}.$$
 Then we take $\ep,\eta>0$ (depending on $\de$) such that
 $$\cos(\b)+\ep\leq -\eta<0.$$

Since $\M$ admits a nonzero proximate order $\ro(r)$, by Remark~\ref{rema.M.and.V.admsibility.condition} there exist $V\in MF(2\o,\ro(r))$
 and positive constants $A,B,t_0$ such that~\eqref{ineq.function.M.and.V} holds.
According to Remark~\ref{rema.bounded.function.null.asymp.expan.dir}, and specifically to~\eqref{eqMFlatDirecBoundedFunction}, there exist $c_1,c_2>0$ with
\begin{equation}\label{ineq.Mflat.one.direction}
|f(z)|\le c_1e^{-M(1/(c_2|z|))},\qquad \arg(z)=\pi\ga/2.
\end{equation}
Choose $d_2>0$ such that $c_2^{-1/\o}>d_2$, and take $a\in\Ri$ with
$$
\arg(a)=\frac{\o \pi}{2}-\frac{\pi\ga}{2}+\frac{\de}{2}, \qquad 0<|a|<\left( \frac{Ad_2}{2}\right)^{\o}.
$$
It is clear that $\ep<1$, so we have that
\begin{equation}
\cos\left(\frac{\arg(a)-\arg(z)}{\o} \right)+\ep\leq 2\label{ineq.cos.every.z}
\end{equation}
for every $z\in\overline{S_\ga}$.

We observe that $\arg(a/z)\in [\o \a, \o\b]\en (-\pi\o/2,\pi\o)$ for every $z\in \overline{S_\ga}$. Taking into account Remark~\ref{rema.order.ro.1.over.omega} and using property (I) of the functions in $MF(2\o,\ro(r))$ we see that
 $$\lim_{|z|\to 0} \frac{V(a/z)}{|a|^{1/\o} V(1/|z|)} = e^{i(\arg(a)-\arg(z))/\o}$$
 uniformly for $\arg(z)\in [- \pi\ga/2, \pi\ga/2 ]$. Consequently,
  $$\lim_{|z|\to 0} \Re\left( \frac{V(a/z)}{|a|^{1/\o} V(1/|z|)} \right)=  \cos((\arg(a)-\arg(z))/\o)$$
uniformly for $\arg(z)\in [- \pi\ga/2, \pi\ga/2 ]$, and we deduce that
\begin{align}
 |a|^{1/\o} V\left(\frac{1}{|z|}\right) \left(  \cos\left( (\arg(a)-\arg(z))/\o\right)  -\ep\right))&\leq
 \Re \left(V\left(\frac{a}{z}\right)\right), \label{inequality.realpart.V.below}\\
 |a|^{1/\o} V\left(\frac{1}{|z|}\right) \left(  \cos\left( (\arg(a)-\arg(z))/\o\right)  +\ep\right)&\geq
 \Re \left(V\left(\frac{a}{z}\right)\right), \label{inequality.realpart.V.above}
\end{align}
for $|z|<s_1$ small enough and  $\arg(z)\in [- \pi\ga/2, \pi\ga/2 ]$.  For convenience, we choose $s_1<1/(t_0c_2)$. Consider the function
$$F(z):=f(z)e^{V(a/z)}.$$
The function $V(a/z)$ is holomorphic in $S(\arg(a),{2\o})\supset\overline{S_{\ga}}$, so $F(z)$
is holomorphic in $S_\ga$ and continuous up to $\partial S_\ga$. Our aim is to apply the \PL\ theorem~\ref{Theo.Prag.Lind} to this function in a suitable bounded sector $S(0,\ga,s_3)$.

If $\arg(z)=-\pi\ga/2$,
we have that $\arg(a)-\arg(z)=\b\o$. Then, since $f$ is bounded in $\overline{S_\ga}$ by a constant $K>0$, by  using~\eqref{inequality.realpart.V.above} we see that for $|z|<s_1$,
$$|F(z)|\leq K e^{\Re(V(a/z))} \leq K e^{(\cos(\b)+\ep)|a|^{1/\o} V(1/|z|)} \leq   K e^{-\eta|a|^{1/\o} V(1/|z|)}.$$
Now, observe that $V(1/|z|)>0$ (property (III)), so we deduce that $|F(z)|\leq K$ for every $z$ with $|z|<s_1$ and $\arg(z)=-\pi\ga/2$.

If $\arg(z)=\pi\ga/2$, we have that $\arg(a)-\arg(z)=\a\o$. Then, from~\eqref{ineq.Mflat.one.direction}, \eqref{ineq.function.M.and.V}, \eqref{ineq.cos.every.z} and~\eqref{inequality.realpart.V.above} we see that, if $|z|<s_1$,
$$|F(z)| \leq c_1 e^{-M(1/(c_2|z|))} e^{(\cos(\a)+\ep)|a|^{1/\o} V(1/|z|)} \leq  c_1 e^{-AV(1/(c_2|z|))+ 2|a|^{1/\o} V(1/|z|)}.$$
Using property (I) of the functions in $MF(2\o,\ro(r))$ we have that
 $$\lim_{|z|\to 0} \frac{V(1/(c_2|z|))}{V(1/|z|)} =c_2^{-1/\o}.$$
Then, for $|z|<s_2\leq s_1$ small enough we have that $V(1/(c_2|z|))\geq d_2 V(1/|z|)$, and we conclude that
$$|F(z)| \leq  c_1 e^{ (-A d_2 + 2 |a|^{1/\o}) V(1/|z|)}, \qquad |z|<s_2,\quad \arg(z)=\pi\ga/2.$$
Since $|a|$ has been chosen small enough in order that $-A d_2 + 2 |a|^{1/\o}<0 $, we deduce that $|F(z)|\leq c_1$ for every $|z|<s_2$ and $\arg(z)=\pi\ga/2$.


For $z\in S_\gamma$ with $|z|<s_1$, by using~\eqref{ineq.cos.every.z} and~\eqref{inequality.realpart.V.above}  we have that
$$ \Re \left(V\left(\frac{a}{z}\right)\right)\leq 2 |a|^{1/\o} V\left(\frac{1}{|z|}\right).$$
As $\ga<\o$, there exists $\mu>0$ such that $\ga<\mu<\o$. By property (VI), we know that $\log(V(t))/\log(t)$ is a proximate order equivalent to $\ro(r)$, hence tending to $1/\o$ at infinity. Then, we can apply Remark~\ref{rema.V.order.less.ro}: there exists $0<s_3\leq s_2$ small enough such that for every $z\in S_\ga $, $|z|\leq s_3$,
$$ \Re \left(V\left(\frac{a}{z}\right)\right)\leq 2 |a|^{1/\o} \left(\frac{1}{|z|}\right)^{1/\mu}.$$
Since $f(z)$ is bounded in $S_\ga$, we have that
$$|F(z)|\leq K \exp ( 2 |a|^{1/\o} |z|^{-1/\mu} ) , \qquad z\in S_\ga,\quad |z|\leq s_3,$$
and, in particular,
$$|F(z)|\leq K \exp ( 2 |a|^{1/\o} s_3^{-1/\mu} ) , \qquad z\in S_\ga,\quad |z|= s_3.$$
By applying  \PL\ theorem~\ref{Theo.Prag.Lind} to the function $F(z)$ in $S(0,\ga,s_3)$, we obtain that
$$
|F(z)|\leq K_0:=\max(K,c_1, K \exp ( 2 |a|^{1/\o} s_3^{-1/\mu} ))
$$
for $|z|\leq s_3 $ and $\arg(z)\in[-\pi\ga/2,\pi\ga/2]$.

Consequently, using~\eqref{inequality.realpart.V.below}, if $|z|\leq s_3$ and $\arg(z)\in[-\pi\ga/2,\pi\ga/2]$ we have that
$$|f(z)|\leq K_0 e^{\Re(-V(a/z))} \leq K_0 e^{-(\cos((\arg(a)-\arg(z))/\o)-\ep)|a|^{1/\o} V(1/|z|)}. $$
Assuming that $\arg(z)\in[-\pi\ga/2+\de,\pi\ga/2]$, we deduce that
$$\cos((\arg(a)-\arg(z))/\o)\geq \cos\left(\frac{\pi}{2}-\frac{\de}{2\o} \right)=-\cos(\b)\geq \eta+\ep >0.$$
Then, for $r_2:=\eta|a|^{1/\o} >0$ we find that for every $z$ with $\arg(z)\in[-\pi\ga/2+\de,\pi\ga/2]$ and $|z|<s_3$ we have that
$$|f(z)|\leq K_0 e^{-r_2 V(1/|z|)}. $$
Choose $k_2>0$ such that $(1/k_2)^{1/\o}<r_2/B$. Property (I) of the functions in $MF(2\o,\ro(r))$ implies that,
for $z$ with $|z|<s_4 <\min(s_3,1/(t_0k_2))$, small enough, and $\arg(z)\in[-\pi\ga/2+\de,\pi\ga/2]$,
we have
$$|f(z)|\leq K_0 e^{-B V(1/(k_2|z|))}\leq K_0 e^{-M(1/(k_2|z|))}.$$
We take $k_1:=K_0 e^{M(1/(k_2 s_4))}\geq K_0$. Then, since $M(t)$ is nondecreasing, if $|z|\geq s_4$ and $\arg(z)\in[-\pi\ga/2+\de,\pi\ga/2]$ we have
$$
|f(z)|\leq K \leq K_0 = k_1 e^{-M(1/(k_2 s_4))}\leq k_1 e^{-M(1/(k_2 |z|))},
$$
which concludes the proof.
\end{proof}

\begin{rema}\label{rema.bisect.direct.Mflat} Some comments are in order concerning the statement or proof of the previous result.

By a simple rotation, one may easily check that the validity of
Lemma~\ref{lemma.Mflat.one.direction.small.opening} does not depend on the bisecting direction of the sector where the function $f$ is defined.
Moreover, one could slightly weaken the hypotheses by considering a function $f$ holomorphic in $S_{\ga}$ that admits a continuous extension to the direction $d=\pi\ga/2$, in which it is $\M$-flat, and that is bounded in every (half-open) sector
$$
\{z\in\mathcal{R}:\arg(z)\in(-\frac{\pi\ga}{2}+\de,\frac{\pi\ga}{2}]\},\quad \de>0.
$$

Indeed, we may give a more precise information about the type. Following the previous proof, one notes that
$$
k_2=k_2(\de)>\Big(\frac{B}{r_2}\Big)^{\o}=
\Big(\frac{B}{\eta |a|^{1/\o}}\Big)^{\o}\ge
\Big(\frac{2B}{Ad_2\cos(\frac{\pi}{2}-\frac{\delta}{2\o})}\Big)^{\o}
\ge \Big(\frac{2B}{A}\Big)^{\o}
\Big(\frac{1}{\sin(\frac{\delta}{2\o})}\Big)^{\o} c_2,$$
and $k_2$ may be made arbitrarily close to the last expression at the price of enlarging the constant $k_1=k_1(\de)$.
So, the original type $c_2$ is basically affected by a precise factor when moving to a direction $\theta=-\pi\ga/2+\de$ with  $0<\de<\pi\ga$. It is obvious that $k_2(\de)$ explodes at least like $1/\sin^\o(\de)$ as $\de\to 0$. This means that the type of the null asymptotic expansion tends to 0 as the direction in the sector approaches the boundary $d=-\pi\ga/2$, in the same way as in the Gevrey case (see Theorem~\ref{theoFruchardZhang}).

Moreover, the constant 2 in $\de/(2\o)$ could be any number greater than 1 and, by suitably choosing the value $\varepsilon$ in the proof, the constant $2B/A$ appearing before can be made as close to $B/A$ as desired, so that the only indeterminacy in the previous factor is caused by the values $A,B$ involved in~\eqref{ineq.function.M.and.V}. In the common situation that the function $d_{\M}(t)$ is indeed a proximate order, the constants $A$ and $B$ can also be taken as near to $1$ as wanted, what makes the expression even more explicit.

Finally, note that, by using Theorem~\ref{theo.charact.admit.p.o} one may change $\M$ by an equivalent sequence $\L$ such that $d_\L$ is a proximate order. However, this fact does not improve the proof,
since again Theorem~\ref{propanalproxorde} will be applied to obtain a function $V\in MF(2\o,d_{\L}(t))$ and we will work with the same type of estimate that we have in~\eqref{ineq.function.M.and.V}.
\end{rema}

The following lemma shows that imposing $\ga<\o(\M)$ is only a technical condition in order to apply \PL\
theorem~\ref{Theo.Prag.Lind}.

\begin{lemma}
\label{lemma.Mflat.one.direction.large.opening}
Let $\M$ and $\ga>0$ be given. Suppose $f$ is a bounded holomorphic function in $S_{\ga}$ that admits a continuous extension to the boundary $\partial S_\ga$, and that is $\M$-flat in direction $d=\pi\ga/2$.
Then for every $0<\de<\pi\ga$, there exist constants $k_1(\de),k_2(\de)>0$ with
 $$|f(z)|\leq k_1e^{-M(1/(k_2|z|))},\qquad \arg(z)\in[-\pi\ga/2+\de,\pi\ga/2].  $$
\end{lemma}

\begin{proof} For simplicity we write $\o=\o(\M)$, and put $\theta_0:= \pi\ga/2$. We can obviously choose a suitable natural number $m$ and directions $\theta_j\in(-\pi\ga/2,\pi\ga/2)$, $j=1,2,\dots, m$, such that
	\begin{align*}
	\theta_j&:=\theta_{j-1}-\pi\o/2, \quad  \theta_j\geq-\pi\ga/2+\de, \quad j=1,\dots, m-1,\\
\theta_m&\in(-\pi\ga/2,-\pi\ga/2+\de),\quad	 \theta_{m-1}-\theta_m<\pi\o/2.
	\end{align*}
	We fix $0<\ep< \pi\o/4$. 
Since $\theta_0-\theta_1+\ep<3\pi\o/4<\pi\o$, we can apply Lemma~\ref{lemma.Mflat.one.direction.small.opening} 
to the function
	$f$ restricted to the sector
	$S_1=\{ z\in\Ri:\, \arg(z)\in [\theta_1-\ep,\theta_0])\}$. We deduce that there exist constants $k_{1,1},k_{2,1}>0$ with
 $$|f(z)|\leq k_{1,1}e^{-M(1/(k_{2,1}|z|))},\qquad \arg(z)\in[\theta_1,\theta_0].  $$
	By recursively reasoning in the sectors
	$$S_j=\{z\in\Ri:\, \arg(z)\in [\theta_j-\ep,\theta_{j-1}])  \},\quad j=2,3,\dots, m-1,$$
and finally in the sector
		$$S_m=\{z\in\Ri:\, \arg(z)\in [\theta_m,\theta_{m-1}])  \},$$
we obtain constants $k_{1,j},k_{2,j}>0$	such that
$$
|f(z)|\leq k_{1,j}e^{-M(1/(k_{2,j}|z|))},\qquad \arg(z)\in[\theta_j,\theta_{j-1}].
$$
It is clear then that for $k_1:=\max_{j}k_{1,j}$ and $k_2:= \max_{j}k_{2,j}$ we have that
	 $$|f(z)|\leq k_1e^{-M(1/(k_2|z|))},\qquad \arg(z)\in[-\pi\ga/2+\de,\pi\ga/2].  $$
\end{proof}

In the next result we impose $\M$-flatness in both boundary directions of the sector, and conclude uniform $\M$-flatness throughout the sector.


\begin{lemma}\label{lemma.Mflat.two.directions}
Let $\M$ and $\ga>0$ be given. Suppose $f$ is a bounded holomorphic function in $S_{\ga}$ that admits a continuous extension to the boundary $\partial S_\ga$, and that is $\M$-flat in directions $d=\pi\ga/2$ and $-d$.
Then there exist constants $k_1,k_2>0$ with
\begin{equation}\label{eqUniformMExponDecrease}
|f(z)|\leq k_1e^{-M(1/(k_2|z|))},\qquad \arg(z)\in[-\pi\ga/2,\pi\ga/2].  \end{equation}
\end{lemma}

\begin{proof}
By Lemma~\ref{lemma.Mflat.one.direction.large.opening},  there exist constants $k_{1,1},k_{2,1},k_{1,2},k_{2,2}>0$ such that
$$|f(z)|\leq k_{1,1}e^{-M(1/(k_{2,1}|z|))},\qquad \arg(z)\in[0,\pi\ga/2]  $$
and
$$|f(z)|\leq k_{1,2}e^{-M(1/(k_{2,2}|z|))},\qquad \arg(z)\in[-\pi\ga/2,0].  $$
We conclude taking $k_1:=\max\{k_{1,1},k_{1,2}\}$ and $k_2:=\max\{k_{2,1},k_{2,2}\}$.
\end{proof}

\begin{rema}\label{rema.bounded.sector.Mflat} By carefully inspecting its proof, we see that Lemma~\ref{lemma.Mflat.one.direction.small.opening} holds true in any bounded sector $S(d,\ga,r)$ and, consequently, Lemma~\ref{lemma.Mflat.one.direction.large.opening} and Lemma~\ref{lemma.Mflat.two.directions} are also valid in bounded sectors.
\end{rema}


We show next that, as Remark~\ref{rema.bounded.sector.Mflat} suggests, it is also possible to work in sectorial regions.

\begin{prop}
\label{lemma.Mflat.one.direction.sectorial.regions}
Let $\M$ and $\ga>0$ be given. Suppose $f$ is holomorphic in a sectorial region $G_{\ga}$, bounded in every $T\ll G$, and $\M$-flat in a direction $\theta$ in $G_{\ga}$.
Then, for every $T\ll G_\ga$ there exist
constants $k_1(T),k_2(T)>0$ with
\begin{equation}\label{ineq.Mflat.subsector}
|f(z)|\leq k_1e^{-M(1/(k_2|z|))},\qquad z\in T.
\end{equation}
\end{prop}

\begin{proof}
By suitably enlarging the opening of the subsector, we can assume that $\theta$ is one of the directions in $T$.
There exist $R,c_1,c_2>0$ with
\begin{equation}\label{ineq.Mflat.one.direction.bounded.ray}
|f(z)|\le c_1e^{-M(1/(c_2|z|))},\qquad \arg(z)=\theta, \qquad|z|\leq R.
\end{equation}
If $\theta_1<\theta_2$ are the (radial) boundary directions of $T$, we consider $\de>0$ such that $-\pi\ga/2<\theta_1-\de$ and  $\theta_2+\de<\pi\ga/2 $. There exists $0<r<R$ such that the sectors $S_1=\{z\in \Ri:\,|z|\leq r,\, \arg(z)\in[\theta_1-\de,\theta]\}$ and $S_2=\{z\in \Ri:\,|z|\leq r, \, \arg(z)\in[\theta,\theta_2+\de]\}$ are contained in $G_\ga$.
Taking into account \eqref{ineq.Mflat.one.direction.bounded.ray} and Remark~\ref{rema.bounded.sector.Mflat}, we can apply Lemma~\ref{lemma.Mflat.one.direction.large.opening}  to the restriction of $f$ to each sector and we conclude that $f$
is $\M$-flat for $\arg(z)\in[\theta_1,\theta_2]$ and $|z|\leq r$. Since $M(t)$ is nondecreasing, by suitably enlarging the constant $k_1$ we obtain~\eqref{ineq.Mflat.subsector}.
\end{proof}

\begin{exam} 
Boundedness of the considered function is necessary in any of the previous results in this section. The next example shows that having an $\M$-asymptotic expansion in a direction $d$ does not guarantee its validity in any sector containing that direction. Our inspiration comes from a similar example in W. Wasow's book~\cite[p.\ 38]{Wasow}, which concerned the function $f(z)=\sin(e^{1/z}) e^{-1/z}$.

Given $\M$, by Remark~\ref{rema.M.and.V.admsibility.condition} for every $\ga>0$ there exists $V\in MF(\ga,\ro(r))$ such that we have~\eqref{ineq.function.M.and.V}.
We consider the function $$f(z)=\sin(e^{V(1/z)}) e^{-V(1/z)}\qquad z\in S_{\ga}.$$
Since $\sin(e^{V(1/z)})$ is bounded for real $z>0$, we see that $f$ is $\M$-flat in direction $0$. If we compute the derivative of $f$ in $S_\ga$ we see that
\begin{align*}
f'(z)&=\frac{V'(1/z)}{z^2} \left(\sin(e^{V(1/z)}) e^{-V(1/z)}-\cos(e^{V(1/z)})\right) \\
&=\frac{V'(1/z)}{z V(1/z)} \frac{V(1/z)}{z}\left(\sin(e^{V(1/z)}) e^{-V(1/z)}-\cos(e^{V(1/z)})\right).
\end{align*}
Since for $z>0$ we have  $\lim_{z\to 0} (1/z)V'(1/z)/V(1/z)=1/\o(\M)$ (by property (VI), see~\cite[Prop.~1.2]{Maergoiz}) and $\lim_{z\to 0} V(1/z)/z=\oo$ (property (III)), we deduce that $\lim_{z\to 0} f'(z)$ does not exist.
By Remark~\ref{rema.limit.f.at.0}, $f$ can not have $\M$-asymptotic expansion in any sectorial region containing direction $0$. Consequently, $f$ is not $\M$-flat in any such sectorial region.
We note that, in particular, the example of Wasow corresponds to the Gevrey case of order 1, i.e., to the sequence $\M=(p!)_{p\in\N_0}$.
\end{exam}

\begin{rema}\label{rema.Beurling.case}
At this point it is worth saying a few words about a situation which, although not usually considered in the theory of asymptotic expansions, plays an important role in the general framework of ultradifferentiable or ultraholomorphic classes, namely that of the so-called Carleman classes of Beurling type. We will not give full details here, but let us say that a function $f$, holomorphic in a sectorial region $G$, has \emph{Beurling $\M$-asymptotic expansion $\hat{f}=\sum_{n=0}^\oo a_n z^n $ in a direction $\theta$} in $G$ if there exists $r_\theta>0$ such that the segment $(0,r_\theta e^{i\theta}]$ is contained in $G$, and for every $A_\theta>0$ (small) there exists $C_\theta>0$ (large) such that
for every $z\in (0,r_\theta e^{i\theta}]$ and every $p\in\N_0$ one has
\begin{equation*}
\Big|f(z)-\sum_{n=0}^{p-1}a_nz^n \Big|\le C_\theta A_\theta^pM_{p}|z|^p.
\end{equation*}
Following the idea in Remark~\ref{rema.bounded.function.null.asymp.expan.dir}, one can prove that $f$, bounded throughout $G$, is Beurling $\M$-flat in direction $\theta$ if, and only if,  for every $c_2>0$ (small) there exist $c_1>0$ (large) such that for every $z\in G$ with $\arg(z)=\theta$ one has
\begin{equation}\label{eqBeurlingMFlatDirecBoundedFunction}
|f(z)|\le c_1 e^{-M(1/(c_2 |z|))}.
\end{equation}
Then, the following analogue of Lemma~\ref{lemma.Mflat.one.direction.small.opening} is valid: Given $\M$ and $0<\ga<\o(\M)$, suppose $f$ is a bounded holomorphic function in $S_{\ga}$ that admits a continuous extension to the boundary $\partial S_\ga$, and that is Beurling $\M$-flat in direction $d=\pi\ga/2$.
Then for every $0<\de<\pi\ga$ and every $k_2>0$, there exists a constant $k_1=k_1(\de,k_2)>0$ such that
 $$|f(z)|\leq k_1e^{-M(1/(k_2|z|))},\qquad \arg(z)\in[-\pi\ga/2+\de,\pi\ga/2].  $$
The proof of this statement follows the same lines as that of the original lemma, by carefully tracing the dependence of the different constants involved in the estimates. Indeed, the constants $A,B,\a,\b,\ep,\eta$ are determined in the same way. Choose $r_2>0$ such that $r_2/B>k_2^{-1/\o}$, and a point $a$ with the specified argument and modulus $(r_2/\eta)^{\o}$. Take a positive $d_2$ such that
$d_2>2|a|^{1/\o}/A$, and then $c_2>0$ such that $c_2<d_2^{-\o}$. By definition of Beurling $\M$-flatness in direction $\ga\pi/2$, there exists $c_1>0$ such that \eqref{eqBeurlingMFlatDirecBoundedFunction} holds for $\arg(z)=\ga\pi/2$. Then, the desired estimates hold for the same $k_1>0$ obtained in the proof of that lemma.\par

Note that also Lemma~\ref{lemma.Mflat.one.direction.large.opening}, Lemma~\ref{lemma.Mflat.two.directions} and Proposition~\ref{lemma.Mflat.one.direction.sectorial.regions} will be valid in this Beurling setting.
\end{rema}

\section{Watson's Lemmas}


We will now obtain several quasianalyticity results by combining those in Subsections~\ref{subsectClassicalQuasianal} and~\ref{subsectQuasianalResultsProxOrders} with the results on the propagation of null asymptotics in Section~\ref{sectionFlatnessExtension}.

\begin{rema}\label{rema.null.uniform.asymp.expan.expon.decay}
In a similar way as in the proof of Theorem~\ref{theo.thillez.flat.function.M} (see~\cite{ThilliezSmooth}), it is easy to deduce that, given a bounded holomorphic function $f$ in a sector $S_{\ga}$ that admits a continuous extension to the boundary $\partial S_\ga$, the fact that $f\in\tilde{\mathcal{A}}_{\M}^u(S_\ga)$ and $f$ is $\M$-flat amounts to the existence of constants $k_1,k_2>0$ such that
\eqref{eqUniformMExponDecrease} holds.
\end{rema}

In the first version, an immediate consequence of previous information, we assume the function is flat at both boundary directions.

\begin{lemma}
\label{Lemma.Watson.two.direction}
Let $\M$ and $\ga>0$ be given, such that either $\ga>\o(\M)$,
or $\ga=\o(\M)$ and $\sum^{\oo}_{p=0} (m_p)^{-1/\o(\M)}$ diverges.
Suppose $f$ is a bounded holomorphic function in $S_{\ga}$ that admits a continuous extension to the boundary $\partial S_\ga$, and that is $\M$-flat in directions $d=\pi\ga/2$ and $-d$.
Then $f\equiv 0$.
\end{lemma}

\begin{proof} By Lemma~\ref{lemma.Mflat.two.directions}, we know that \eqref{eqUniformMExponDecrease} holds for suitable $k_1,k_2>0$. The previous remark implies that
$f\in\tilde{\mathcal{A}}^u_{\M}(S_{\ga})$ and  $f\sim_{\M}\hat{0}$, and by Corollary~\ref{corollary.Watson.Lemma.Uniform.Bounds} we deduce that $f\equiv0 $.
\end{proof}

In the second, improved version we assume only that the function is flat in one of the boundary directions.

\begin{lemma}
\label{Lemma.Watson.one.direction} 
Assume the same hypotheses as in Lemma~\ref{Lemma.Watson.two.direction}, except that now $f$ is $\M$-flat only in direction $d=\pi\ga/2$.
Then $f\equiv 0$.
\end{lemma}

\begin{proof}
For simplicity we write $\o=\o(\M)$. The argument is simple if $\ga>\o$: We fix $\o<\mu<\ga$ and $\de=(\ga-\mu)\pi>0$. By Lemma~\ref{lemma.Mflat.one.direction.large.opening}
we know that there exist constants $k_1(\de),k_2(\de)>0$ with
\begin{equation*}
 |f(z)|\leq k_1e^{-M(1/(k_2|z|))},\qquad \arg(z)\in[\pi\ga/2-\mu\pi,\pi\ga/2].
\end{equation*}
Then, Remark~\ref{rema.null.uniform.asymp.expan.expon.decay} implies that
$f\in\tilde{\mathcal{A}}^u_{\M}(S)$, with $S=\{z\in \Ri:\, \arg(z)\in (\pi\ga/2-\mu\pi, \pi\ga/2)\}$ and  $f\sim_{\M}\hat{0}$. Since $\mu>\o$, we can apply Corollary~\ref{corollary.Watson.Lemma.Uniform.Bounds} to the function $f$ in $S$ (see also Remark~\ref{rema.Watson.Lemma.bisect.direct}),
 and we deduce that $f\equiv 0$.

If $\ga=\o$ we fix $\de=\pi \o/8>0$. Lemma~\ref{lemma.Mflat.one.direction.large.opening} ensures there exist $k_1(\de),k_2(\de)>0$ with
\begin{equation}\label{ineq.Mflat.almost.all.sector}
 |f(z)|\leq k_1e^{-M(1/(k_2|z|))},\qquad \arg(z)\in[-3\pi\o/8,\pi\o/2].
\end{equation}
As in the proof of Lemma~\ref{lemma.Mflat.one.direction.small.opening}, since $\M$ admits a nonzero proximate order $\ro(r)$, there exist $V\in MF(2\o,\ro(r))$ and positive constants $A,B,t_0$ such that we have~\eqref{ineq.function.M.and.V}.
Choose $q_2>0$ such that $k_2^{-1/\o}>q_2$, and take $a\in\Ri$ such that
$$
\arg(a)=\frac{\o \pi}{4}, \qquad 0<|a|<\left( \frac{A q_2}{2}\right)^{\o}.
$$
We observe that for every $z$ with $\arg(z) \in[-\pi\o/2,\pi\o/2 ]$ one has
$$
\arg(a/z)\in [-\pi\o/4,3\pi\o/4] \en (-\pi\o/2,\pi\o).
 $$
Using property (I) of the functions in $MF(2\o,\ro(r))$, we see that
  $$\lim_{|z|\to 0} \Re\left( \frac{V(a/z)}{|a|^{1/\o} V(1/|z|)} \right)=  \cos((\arg(a)-\arg(z))/\o)$$
uniformly for $\arg(z)\in [-\pi\o/2,\pi\o/2]$. We fix $0<\ep<1$ such that
$$\cos(3\pi/4) + \ep\leq\cos(5\pi/8)+\ep \leq -1/3 <0.$$
We deduce that we have~\eqref{inequality.realpart.V.below} and~\eqref{inequality.realpart.V.above}
for $\arg(z)\in [-\pi\o/2,\pi\o/2]$  and $|z|<s_1$, small enough and subject to the restriction $s_1<1/(t_0k_2)$. Consider the function
$$F(z):=f(z)e^{V(a/z)}, \qquad \arg(z) \in[-\pi\o/2,\pi\o/2].$$
Then we see that $F(z)$ is holomorphic in $S_{\o}$
and continuous in $\overline{S_{\o}}$.

If $\arg(z)\in[-\pi\o/2, -3\pi\o/8]$,
we have that $\arg(a/z)\in[5\pi\o/8,3\pi\o/4]$. Then, since $f(z)$ is bounded by $K>0$ in $\overline{S_\o}$ and
  using~\eqref{inequality.realpart.V.above} for $|z|<s_1$, one has
$$|F(z)|\leq K e^{\Re(V(a/z))} \leq K e^{(\cos(5\pi/8)+\ep)|a|^{1/\o} V(1/|z|)} \leq   K e^{-|a|^{1/\o} V(1/|z|)/3}.$$
Using property (I) of the functions in $MF(2\o,\ro(r))$ we see that
 $$\lim_{|z|\to 0} \frac{V((|a|/(3B)^\o)(1/2|z|))}{(|a|^{1/\o}/(3B)) V(1/|z|)} = (1/2)^{1/\o}<1.$$
We define $b_2:=(|a|/(3B)^\o)/2$. Then for $|z|<s_2<\min(s_1,b_2/t_0)$, small enough, we have that
$$|F(z)|\leq K e^{-BV(b_2/|z|)} , \qquad |z|<s_2, \qquad \arg(z)\in[-\pi\o/2, -3\pi\o/8].$$
Using~\eqref{ineq.function.M.and.V}, we see that
\begin{equation}\label{eq.F.Mflat.little.z}
|F(z)|\leq K e^{- M(b_2/|z|)} , \qquad |z|<s_2, \qquad \arg(z)\in[-\pi\o/2, -3\pi\o/8].
\end{equation}
We define $ C= \max\{\Re(V(a/z)):\, |z|\geq s_2, \, -\pi\o/2 \leq\arg(z)\leq  -3\pi\o/8\} $ and we take
$$c_1:=K \max \{\exp(C), 1\}<\oo.$$
 Then, since $M(t)\geq 0$ we have that
\begin{equation}\label{eq.F.Mflat.big.z}
|F(z)|\leq c_1 \leq c_1 e^{M(b_2/|z|)} \qquad |z|\geq s_2, \qquad \arg(z) \in[-\pi\o/2, -3\pi\o/8].
\end{equation}
Since $c_1\geq K$, from~\eqref{eq.F.Mflat.little.z} and~\eqref{eq.F.Mflat.big.z} we deduce that $F$ is $\M$-flat for $\arg(z)\in[-\pi\o/2, -3\pi\o/8]$.

If $\arg(z)\in[-3\pi\o/8, \pi\o/2]$, we have that $\arg(a/z)\in[-\pi\o/4,5\pi\o/8]$. Using~\eqref{ineq.function.M.and.V},
~\eqref{inequality.realpart.V.above}   and
~\eqref{ineq.Mflat.almost.all.sector}, for $|z|<s_1$ we see that
$$|F(z)| \leq k_1 e^{-M(1/(k_2|z|))} e^{(\cos(\arg(a/z)/\o)+\ep)|a|^{1/\o} V(1/|z|)} \leq
k_1 e^{-AV(1/k_2|z|)+2|a|^{1/\o} V(1/|z|)}.$$
Now, property (I) of the functions in $MF(2\o,\ro(r))$ lets us write
 $$\lim_{|z|\to 0} \frac{V(1/k_2|z|)}{V(1/|z|)} =k_2^{-1/\o},$$
so that, for $|z|<s_3\leq s_2$ small enough,  we have that $V(1/k_2|z|)\geq q_2 V(1/|z|)$. We conclude that
$$|F(z)| \leq k_1 e^{ (-A q_2 +2 |a|^{1/\o}) V(1/|z|)}, \qquad |z|<s_3,\quad \arg(z)\in[-3\pi\o/8, \pi\o/2].$$
Since $|a|$ has been chosen small enough in order that $-A q_2 +2 |a|^{1/\o}<0 $,
proceeding as before, we find that $F(z)$ is $\M$-flat for $\arg(z)\in[-3\pi\o/8, \pi\o/2]$.

Consequently, $F$ verifies estimates of the type~\eqref{eqUniformMExponDecrease} in $\overline{S_{\o}}$
and, by Remark~\ref{rema.null.uniform.asymp.expan.expon.decay},
$F\in\tilde{\mathcal{A}}^u_{\M}(S_\o)$ and  $F\sim_{\M}\hat{0}$. Since  $\sum^{\oo}_{p=0} (m_p)^{-1/\o(\M)}$  is assumed to be divergent, we can apply Corollary~\ref{corollary.Watson.Lemma.Uniform.Bounds} to the function $F(z)$  in $S_\o$, and deduce that $F(z)\equiv 0$ and $f\equiv 0$.
\end{proof}

In the proof of Lemma~\ref{Lemma.Watson.one.direction} we need to distinguish two situations: in case $\ga>\o(\M)$ we have been given an $\M$-flat function $f$ in a wide enough sector (what entails uniqueness), while in case $\ga=\o(\M)$ an $\M$-flat function $F$ in a sector of opening $\pi\o(\M)$ has to be constructed in order to apply  Corollary~\ref{corollary.Watson.Lemma.Uniform.Bounds}, what is possible thanks to the additional assumption on the series $\sum^{\oo}_{p=0} (m_p)^{-1/\o(\M)}$.

It is interesting to note that in the Gevrey case the aforementioned series diverges, so that the previous result extends Lemma 5 in~\cite{FruchardZhang}. Indeed, in that instance the very divergence of the series allows one to treat the case  $\ga>\o(\M)$ by restricting the function to a sector with  $\ga=\o(\M)$, an argument which is not available in our situation.

\begin{rema}\label{rema.partial.version}
In most situations we can obtain converse statements to Lemma~\ref{Lemma.Watson.two.direction} and Lemma~\ref{Lemma.Watson.one.direction}. Observe that if $\ga<\o(\M)$ and we take $\ga<\mu<\o(\M)$, by Corollary~\ref{corollary.Watson.Lemma.Uniform.Bounds} we know there exists a nontrivial $\M$-flat function
$f\in \tilde{\mathcal{A}}^u_{\M}(S_\mu)$. Then (the restriction of) $f$ is a bounded holomorphic function in $S_{\ga}$ that admits a continuous extension to the boundary $\partial S_\ga$, and that is $\M$-flat in directions $d=\pi\ga/2$ and $-d$.\par
Analogously, if $\ga=\o$ and $\sum^{\oo}_{p=0} ((p+1)m_p)^{-1/(\o(\M)+1)}$ converges, we deduce that $\sum^{\oo}_{p=0} (m_p)^{-1/\o(\M)}$ converges too. So, by
Corollary~\ref{Coro.Salinas} there exists a nontrivial $\M$-flat function $f\in \mathcal{A}_{\M}(S_{\o(\M)})$. Since the derivatives of $f$
are Lipschitzian, one may continuously extend $f$ to the boundary of $S_{\o(\M)}$ preserving the estimates, and again obtain that $f$ is $\M$-flat in directions $\pi\o(\M)/2$ and $-\pi\o(\M)/2$.

However, the converse of Lemma~\ref{Lemma.Watson.two.direction} and Lemma~\ref{Lemma.Watson.one.direction} fails in case $\ga=\o(\M)$, the series $\sum^{\oo}_{p=0} (m_p)^{-1/\o(\M)}$ converges and
$\sum^{\oo}_{p=0} ((p+1)m_p)^{-1/(\o(\M)+1)}$ diverges (for instance, this is the situation for the sequence $\M_{1,3/2}$, see the Examples~\ref{exampleSequences}(i)). Although nontrivial $\M$-flat functions in $\tilde{\mathcal{A}}^u_{\M}(S_{\o(\M)})$ exist in this situation, there is no warranty that they can be continuously extended to the boundary of the sector.
\end{rema}


Finally, we provide a version of Watson's Lemma for functions in  sectorial regions which are flat in a direction.

\begin{prop}
Let $\M$ and $\ga>0$ be given with $\ga>\o(\M)$.
Suppose $f$ is holomorphic in a sectorial region $G_{\ga}$, bounded in every $T\ll G$, and $\M$-flat in a direction $\theta$ in $G_{\ga}$.
Then $f\equiv 0$.
\end{prop}

\begin{proof}
Using Proposition~\ref{lemma.Mflat.one.direction.sectorial.regions} we know that for every $T\ll G_\ga$ we have \eqref{ineq.Mflat.subsector} for suitable $k_1,k_2>0$ depending on $T$ and for every $z\in T$.
Then, Theorem~\ref{theo.thillez.flat.function.M} implies that
$f\in\tilde{\mathcal{A}}_{\M}(G_\ga)$ and  $f\sim_{\M}\hat{0}$, and Theorem~\ref{TheoWatsonlemmaProxOrder} leads to the conclusion.
\end{proof}

\begin{rema}
By Theorem~\ref{TheoWatsonlemmaProxOrder}, if $\ga\leq\o$ we can find a nontrivial function $f\in\tilde{\mathcal{A}}_{\M}(G_\ga)$ such that  $f\sim_{\M}\hat{0}$, so it is bounded on every proper bounded subsector $T$ of $G_\ga$ and $\M$-flat in any direction $\theta_0\in(-\pi\ga/2,\pi\ga/2)$. Consequently, in this situation we have a complete version of Watson's Lemma.
\end{rema}

\section{Asymptotic expansion extension}

The next result (see~\cite[Theorem\ 6.1]{SanzFlat}) was stated for strongly regular sequences $\M$ such that $d_{\M}$ is a proximate order. However, as it is deduced from~\cite[Remark 4.11.(iii)]{SanzFlat} and \cite[Remark~4.15]{JimenezSanzSchindl}, it is enough to ask for the sequence to satisfy our two general assumptions (see Section~\ref{sectionFlatnessExtension}).

\begin{theo}[Generalized Borel--Ritt--Gevrey theorem]\label{theoBorelRittGevreySanz}
Let $\M$ and $\ga>0$ be given. The following statements are equivalent:
\begin{itemize}
\item[(i)] $\gamma\le\omega(\M)$,
\item[(ii)] For every $\hat{f}=\sum_{p\in\N_0} a_p z^p\in\C[[z]]_{\M}$ there exists a function $f\in\tilde{\mathcal{A}}_{\bM}(S_{\gamma})$ such that $$f\sim_{\M}\hat{f},$$
    i.e., $\tilde{\mathcal{B}}(f)=\hat{f}$. In other words, the Borel map $\tilde{\mathcal{B}}:\tilde{\mathcal{A}}_{\M}(S_{\gamma})\longrightarrow \C[[z]]_{\M}$ is surjective.
\end{itemize}
\end{theo}

From this result we may generalize Theorem 1 in~\cite{FruchardZhang}.


\begin{theo}
Given $\M$ and $\ga>0$, suppose $f$ is holomorphic in a sectorial region $G_\ga$ is bounded in every $T\ll G_\ga$, and it admits $\hat{f}\in\C[[z]]$ as its $\M$-asymptotic expansion in a direction $\theta\in(-\pi\ga/2,\pi\ga/2)$. Then, $f\in \tilde{\mathcal{A}}_{\M}(G_\ga)$ and $f\sim_{\M}\hat{f}$ in $G_\ga$.
\end{theo}

\begin{proof}
We distinguish two cases:
\begin{enumerate}
	\item Sectorial regions of small opening: If $\ga<\o$, we take $\ga<\mu<\o$. By the Borel-Ritt-Gevrey Theorem~\ref{theoBorelRittGevreySanz} we know that there exists a function $f_0\in \tilde{\mathcal{A}}_{\M}(S_\mu)$ such that $f_0\sim_{\M}\hat{f}$ in $S_\mu$. Then the function
$g:=f-f_0$ is holomorphic in $G_\ga$, bounded in every proper bounded subsector of $G_\ga$ and it is $\M$-flat in direction $\theta$.
Using Proposition~\ref{lemma.Mflat.one.direction.sectorial.regions}, we see that $g$ is $\M$-flat in $G_\ga$.

Then, for every proper bounded subsector $T$ of $G_\ga$, there exists positive constants $A(T),B(T),C(T),D(T)>0$ such that
\begin{align*}
|f(z)-\sum_{n=0}^{p-1} a_n z^n| &\leq |g(z)|+|f_0(z)-\sum_{n=0}^{p-1} a_n z^n| \\
&\leq
AC^pM_p|z|^p+BD^pM_p|z|^p\le 2 \max (A, B) \max(C^p,D^p)  M_p|z|^p \end{align*}
for every $z\in T$ and every $p\in\N_0$. Consequently, $f\in \tilde{\mathcal{A}}_{\M}(G_\ga)$ and $f\sim_{\M}\hat{f}$ in $G_\ga$.

\item Sectorial regions of large opening: If $\ga\geq\o$, we may choose natural numbers $\ell$ and $m$, and for $j=-\ell,\dots,-1,0,1,2,\dots, m$ we may consider directions $\theta_j\in(-\pi\ga/2,\pi\ga/2)$
 such that
	\begin{align*}
	\theta_0:= \theta,\qquad \theta_j:=&\theta_{j-1}+\pi\o/8,& \quad j=1,\dots, m, \qquad \pi\ga/2-\theta_m&<\pi\o/8.\\
	\qquad \theta_{j}:=&\theta_{j+1}-\pi\o/8,& \quad j=-1,\dots, -l,\qquad -\pi\ga/2+\theta_{-l}&>- \pi\o/8 .\\
	\end{align*}
There exists $\ro_0>0$ such that $S_0=S(\theta_0,\pi\o/4,\ro_0)\en G_\ga$. We apply the first part in the sector $S_0$ and we see that
$f\in \tilde{\mathcal{A}}_{\M}(S_0)$ and $f\sim_{\M}\hat{f}$ in $S_0$. In particular, $f$ admits $\hat{f}$ as its $\M$-asymptotic expansion in directions $\theta_1$ and $\theta_{-1}$ for $|z|<\ro_0$. Repeating the process we see that
$f\in \tilde{\mathcal{A}}_{\M}(G_\ga)$ and $f\sim_{\M}\hat{f}$ in $G_\ga$.
\end{enumerate}
\end{proof}

The proof of our last statement is now straightforward.

\begin{coro}
Given $\M$, $\ga>0$ and $\theta\in (-\pi\ga/2,\pi\ga/2)$, we have that
\begin{align*}
\tilde{\mathcal{A}}_{\M}(G_\ga) = \{f\in \mathcal{H}(G_\ga):\, &\text{$f$ is bounded in every proper bounded subsector $T$ of $G_\ga$} \\
&\text{and $f$ admits $\M$-asymptotic expansion in direction $\theta$} \}.
\end{align*}
\end{coro}

\noindent\textbf{Acknowledgements}: The first two authors are partially supported by the Spanish Ministry of Economy and Competitiveness under project MTM2016-77642-C2-1-P. The first author is partially supported by the University of Valladolid through a Predoctoral Fellowship (2013 call) co-sponsored by the Banco de Santander. The third author is supported by the FWF-Project J 3948--N35. Part of these results were obtained in the course of a research stay of the first author in the Centro di Ricerca Matematica Ennio De Giorgi (Scuola Normale Superiore di Pisa, Italy), what he expresses his gratitude for.

\vskip.5cm
\noindent Authors' Affiliation:\par\vskip.5cm
Javier Jim\'enez-Garrido and Javier Sanz\par
Departamento de \'Algebra, An\'alisis Matem\'atico, Geometr{\'\i}a y Topolog{\'\i}a\par
Instituto de Investigaci\'on en Matem\'aticas de la Universidad de Valladolid, IMUVA\par
Facultad de Ciencias\par
Universidad de Valladolid\par
47011 Valladolid, Spain\par
E-mail: jjjimenez@am.uva.es (Javier Jim\'enez-Garrido), jsanzg@am.uva.es (Javier Sanz).
\par\vskip.5cm

Gerhard Schindl\par
Fakult\"at f\"ur Mathematik\par
Universit\"at Wien\par
Oskar-Morgenstern Platz 1, A-1090 Wien, Austria\par
E-mail: gerhard.schindl@univie.ac.at.

\end{document}